\documentclass[letterpaper, 10 pt, conference]{ieeeconf}  

\IEEEoverridecommandlockouts                              

\usepackage{amsmath,amssymb,amsfonts}
\usepackage{cuted}
\usepackage{flushend}
\usepackage{mathtools}
\mathtoolsset{showonlyrefs}
\usepackage{cite}

\newtheorem{theorem}{Theorem}[section]

\newtheorem{assumption}{Assumption}[section]

\newtheorem{lemma}[theorem]{Lemma}

\title{\LARGE \bf
Stealthy Cyber-Attack Design Using Dynamic Programming}

\author{Sribalaji C. Anand$^{1}$ and Andr\'e M. H. Teixeira$^{2}$
\thanks{*This work is supported by the Swedish Research Council under the grant 2018-04396 and by the Swedish Foundation for Strategic Research.}
\thanks{$^{1}$ Sribalaji C. Anand is with the Department of Electrical Engineering, Uppsala University, PO Box 65, SE-75103, Uppsala, Sweden. {\tt\small sribalaji.anand@angstrom.uu.se}}%
\thanks{$^{2}$ Andr\'e M. H. Teixeira is with the Department of Information Technology, Uppsala University, PO Box 337, SE 75105, Uppsala, Sweden. {\tt\small andre.teixeira@it.uu.se}}}
\begin{document}

\maketitle
\thispagestyle{empty}
\pagestyle{empty}

\begin{abstract}
This paper addresses the issue of data injection attacks on control systems. We consider attacks which aim at maximizing system disruption while staying undetected in the finite horizon. The maximum possible disruption caused by such attacks is formulated as a non-convex optimization problem whose dual problem is a convex semi-definite program. We show that the duality gap is zero using S-lemma. To determine the optimal attack vector, we formulate a soft-constrained optimization problem using the Lagrangian dual function. The framework of dynamic programming for indefinite cost functions is used to solve the soft-constrained optimization problem and determine the attack vector. Using the Karush–Kuhn–Tucker conditions, we also provide necessary and sufficient conditions under which the obtained attack vector is optimal to the primal problem. Finally, we illustrate the results through numerical examples.
\end{abstract}
\section{INTRODUCTION}\label{sec_intro}

Research in the security of industrial control systems (ICSs) has received considerable attention due to increased cyber-attacks \cite{dibaji2019systems}. The research focus can be largely classified into $(a)$ attack modelling, $(b)$ attack detection, $(c)$ attack impact assessment and $(d)$ attack treatment and prevention \cite{chong2019tutorial}. This paper focuses on the attack impact assessment problem for false data injection (FDI) attacks. 

Existing literature has discussed different aspects of impact assessment in the finite horizon. The necessary and sufficient condition for the attack impact to be bounded was formulated in \cite{teixeira2013quantifying} as a function of the system (block) Toeplitz matrix (TM). Since TM involves higher powers of the system matrices, its computation is prone to numerical errors. To this end, we provide necessary and sufficient conditions for the attack impact to be bounded without using TM.

The paper \cite{umsonst2017security} determines the highest possible attack impact (infinity norm of the attacked system states) when the attack is constrained to be stealthy. Our paper differs from \cite{umsonst2017security} in the facts that $(a)$ we consider the output-to-output gain (OOG) as an impact metric which is advantageous \cite{anandjoint}, and $(b)$ \cite{umsonst2017security} solves for the impact as the maximum of $N_hn$ convex optimization problems (here $N_h$ is the horizon length and $n$ is the order of the system), whereas we propose a single convex optimization problem for impact assessment.

An FDI attack from an optimal control framework was proposed in \cite{wu2018optimal}. In particular, it proposes a switching policy for the adversary to increase the efficacy of the attack when the number of actuators that can be attacked simultaneously is limited. However, the attacks are not stealthy, rather a bound is imposed on the attack energy. 

The authors in \cite{chen2017cyber} and \cite{djouadi2015finite} consider bounded actuator attacks which is similar to a classical $H_{\infty}$ metric based approach to attack design. However, we have shown previously in \cite{anandjoint} that the OOG based design can outperform $H_{\infty}$ based design in the infinite horizon. In \cite{wu2020zero}, the authors do not consider stealthy attacks and in \cite{fotiadis2020constrained}, the authors consider a perfectly undetectable attack whereas we consider FDI attacks in actuators and sensors but not both. Additionally, all of these works assume that the weighting matrices of the cost function are positive (semi-)definite. In this paper, we show that this is a restrictive assumption.

In \cite{an2017data}, the authors considers an adversary that maximizes the disruption whilst remaining stealthy. A data-driven adaptive DP algorithm was proposed for stealthy attack design, whereas we alternatively adopt the framework of \cite{ferrante2015note}. To this end, we present the following contributions. 
\begin{enumerate}
    \item Firstly, the worst-case impact caused by an FDI attack on the sensor or actuator channels in the finite horizon is posed as a non-convex optimization problem. It is solved through its convex dual problem, addressing the limitation of \cite{umsonst2017security}. It is shown using S-lemma that the duality gap is zero.
    \item Secondly, a soft-constrained optimization problem is formulated using the Lagrange dual function, to determine the optimal attack vector. We observe that the weight matrix of the cost function is indefinite. So, we adopt the recently proposed DP approach \cite{ferrante2015note}.
    \item Thirdly, we provide the necessary and sufficient conditions under which the attack vector is optimal (and consequently is stealthy) to the primal problem, thus partially addressing the limitation of \cite{wu2018optimal}. We outline the merits of using the framework of DP by providing insights into how the attack impact can be made bounded by means of state-feedback policies instead of using open-loop attacks computed from TM. 
    \item Finally, this paper serves as a practical application of \cite{ferrante2015note} to the security of ICS.
\end{enumerate}
\textit{Outline:} The problem is formulated in Section \ref{sec_PF}. Section \ref{sec_impact} proposes an optimization framework for determining the attack impact. Section \ref{sec_policy} adopts the DP framework to determine the optimal attack policy. It also discusses the merits of the DP framework. Section \ref{sec_example} provides a numerical illustration of the proposed optimizations frameworks. Finally, we provide concluding remarks in Section \ref{sec_conclusion}.

\section{Problem background}\label{sec_PF}

In this section, we describe the control system structure and the goal of the adversary. Consider the general description \cite{milovsevic2020security} of a finite-horizon discrete-time (DT) linear time-invariant (LTI) system with a process ($\mathcal{P}$), output feedback controller ($\mathcal{C}$) and an anomaly detector ($\mathcal{D}$) as shown in Fig. \ref{System}. The closed-loop system is represented by
\begin{align}
    \mathcal{P}: & \left\{
                \begin{array}{ll}
                    x_p[k+1] &= Ax_p[k] + B \tilde{u}[k]\\
                    y[k] &= Cx_p[k]\\
                    y_p[k] &= C_Jx_p[k] +D_J \tilde{u}[k]
                \end{array}
                \right. \label{P}\\
    \mathcal{C}: & \left\{
                 \begin{array}{ll}
                     z[k+1] &= A_cz[k]+ B_c\tilde{y}[k]\\
                     u[k] &= C_cz[k] + D_c\tilde{y}[k]
                \end{array}
                \right. \label{C}  \\
    \mathcal{D}: & \left\{
                \begin{array}{ll}
                s[k+1] &= A_es[k] +B_e u[k] + K_e \tilde{y}[k]\\
                y_r[k] &= C_es[k] +D_eu[k] + E_e \tilde{y}[k]
                \end{array}
                \right. \label{D}
\end{align}
$k=0,\dots,N_h$, where the state of the process, controller and the observer are represented by $x_p[k] \in \mathbb{R}^{n_x}, z[k] \in \mathbb{R}^{n_z}$ and $s[k] \in \mathbb{R}^{n_s}$ respectively. The control signal generated by the controller and applied to the actuator is denoted by $\tilde{u}[k] \in \mathbb{R}^{n_u}$, and $ u[k] \in \mathbb{R}^{n_u}$ respectively. The measurement output produced by the process is $y[k] \in \mathbb{R}^{n_m}$ , $\tilde{y}[k] \in \mathbb{R}^{n_m}$ is the measurement signal received by the controller and the detector, $y_p[k] \in \mathbb{R}^{n_p}$ is the virtual performance output, $y_r[k] \in \mathbb{R}^{n_r}$ is the residue generated by the detector. The closed-loop system described above is said to have a good performance over the horizon $N_h$, when the energy of the performance output ($||y_p||_{\ell_2, [0,N_h]}^2$) is small and an anomaly is said to be detected when the detector output energy ($||y_r||_{\ell_2, [0,N_h]}^2$) is greater than a predefined threshold, say $\epsilon_r$. Without loss of generality, we assume $\epsilon_r \triangleq 1$. 

\begin{figure}
    \centering
    \includegraphics[width=8.4cm]{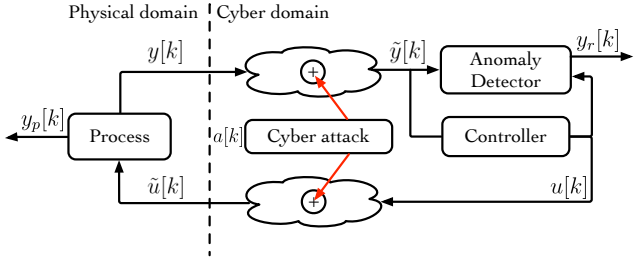}
    \caption{Control system under data injection attack}
    \label{System}
\end{figure}

\subsection{Data injection attack scenario}\label{Attack scenario:sec}
In the closed-loop system described in the previous section, we consider an adversary injecting false data into the sensors or actuators of the process. We next discuss the resources the adversary has access to.
\subsubsection{Disruption and disclosure resources}\label{disclosure:sec} 
The adversary can inject data into the sensor or actuators channels but not both (we do not consider a covert attack). This is represented by
\[ 
\begin{bmatrix}
\tilde{u}[k]\\
\tilde{y}[k]
\end{bmatrix} 
=
\begin{bmatrix}
{u}[k]\\
{y}[k]
\end{bmatrix} 
+
\begin{bmatrix}
B_a\\
D_a
\end{bmatrix} 
a[k],\] 
where $ \begin{bmatrix} B_a\\
D_a
\end{bmatrix} = \begin{bmatrix}
\Gamma_u & 0_{n_u \times n_u} \\
0_{n_m \times n_m}  & \Gamma_y
\end{bmatrix},$ and $a[k] \in \mathbb{R}^{n_a}$ is the data injected by the adversary. In general the adversary cannot access all controller/sensor channels but only a limited number of them which is captured by the rank of the matrix $\begin{bmatrix}B_a^T & D_a^T\end{bmatrix}^T$. The adversary does not have access to any disclosure (eavesdropping) resources.
\subsubsection{System knowledge} 
The adversary knows the closed-loop system. This knowledge is used by the adversary to construct the optimal attack policy. Defining ${x}[k] \triangleq [ x_p[k]^T \; z[k]^T \; s[k]^T]^T$, the closed-loop system under attack with the performance output and detection output as system outputs becomes
\begin{equation}\label{system:CL:uncertain}
                \begin{array}{ll}
                            {x}[k+1] &= {A}_{cl}{x}[k] + {B}_{cl}a[k]\\
                            y_p[k] &= {C}_p{x}[k] + {D}_p a[k]\\
                            y_r[k] &= {C}_r{x}[k] + {D}_r a[k],\\
                \end{array}
\end{equation}
and the closed-loop system matrices are given by
\begin{align*}
   {A}_{cl} &\triangleq \begin{bmatrix}
    A+BD_cC & BC_c & 0\\
    B_cC & A_c & 0\\
    (B_eD_c +K_e)C & B_eC_c & A_e
    \end{bmatrix}, \\
    {B}_{cl} &\triangleq \begin{bmatrix}
    BB_a + BD_cD_a \\ B_cD_a \\ (B_eD_c+K_e)D_a 
    \end{bmatrix},\\
    {C}_p &\triangleq \begin{bmatrix}
    C_J+D_JD_cC & D_JC_c & 0
    \end{bmatrix},\\
    {D}_p &\triangleq D_J(D_cD_a+B_a),\\
    {C}_r &\triangleq \begin{bmatrix}
    (D_eD_c + E_e)C & D_eC_c & C_e
    \end{bmatrix},\\
    {D}_r &\triangleq (D_eD_c +E_e)D_a.
\end{align*}
\subsubsection{Attack goals and constraints.}
Given the resources the adversary has access to, the adversary aims at disrupting the system's behavior while staying stealthy. The system disruption is evaluated by the increase in energy of performance output whereas, the adversary is stealthy if the energy of the detection output is below the threshold $\epsilon_r$. This leads to the optimal attack policy discussed next.
\subsection{Problem formulation}
From the previous discussions, it can be understood that the goal of the adversary is to maximize the performance cost while staying undetected. The attack policy of the adversary can be formulated as
\begin{equation}\label{opti_oog_primal}
\begin{aligned}
\gamma^*  \triangleq \sup_{a\in \ell_{2e,[0,N_h]}} \; & \Vert y_p \Vert_{\ell_2,[0,N_h]}^2 \\
\textrm{s.t.} \quad & \Vert y_r \Vert_{\ell_2,[0,N_h]}^2 \leq 1, \; x[0]=0,
\end{aligned}
\end{equation}
where $\gamma^*$ is the disruption caused by the attack signal on the system and $N_h$ is the horizon length. In the above optimization problem, the constraint $x[0]=0$ is introduced since the system is at equilibrium before the attack. 
\begin{assumption}\label{assume_equil_begin}
The closed-loop system \eqref{system:CL:uncertain} is at equilibrium $x[0]=0$ before the attack commences. $\hfill \triangleleft$
\end{assumption}

The optimization problem \eqref{opti_oog_primal} is non-convex since it has a convex objective function (which has to be maximized) with convex constraints. Thus, in the remainder of this paper, we propose methods to solve \eqref{opti_oog_primal}. In Section \ref{sec_impact}, we determine the optimal value of the optimization problem \eqref{opti_oog_primal}. Whereas in Section \ref{sec_policy}, we determine the optimal attack policy.

\section{Optimization framework for attack impact assessment}\label{sec_impact}

In this section, we determine the attack impact $\gamma^*$, which is the value of the optimization problem \eqref{opti_oog_primal}. Let us define $\textbf{a} \triangleq \begin{bmatrix} a[0]^T,\dots,a[N_h]^T \end{bmatrix}^T$, $\textbf{y}_\textbf{r} \triangleq \begin{bmatrix} y_{r}[0]^T,\dots,y_{r}[N_h]^T \end{bmatrix}^T$, and $ \textbf{y}_\textbf{p} \triangleq \begin{bmatrix} y_{p}[0]^T,\dots,y_{p}[N_h]^T \end{bmatrix}^T$.  Let us additionally define the matrices $\mathcal{T}_{r} \in \mathbb{R}^{n_r(N_h+1) \times n_u(N_h+1)}$, and $\mathcal{T}_{p} \in \mathbb{R}^{n_p(N_h+1) \times n_u(N_h+1)}$ similar to \cite[$(12)$]{teixeira2013quantifying} such that $\textbf{y}_\textbf{r} = \mathcal{T}_{r}\textbf{a}$ and $\textbf{y}_\textbf{p} = \mathcal{T}_{p}\textbf{a}$. Under these definitions, \eqref{opti_oog_primal} can be equivalently written as

\begin{equation}\label{oog_toep}
\begin{aligned}
\gamma^* = \sup_{\textbf{a}} \; &  \textbf{a}^T\mathcal{T}_p^T\mathcal{T}_p\textbf{a} \\
\textrm{s.t.} \quad &  \textbf{a}^T\mathcal{T}_r^T\mathcal{T}_r\textbf{a} \leq 1,\; x[0]=0.
\end{aligned}
\end{equation}

In the following theorem, we propose a convex dual SDP to solve the optimization problem \eqref{oog_toep}.
\begin{theorem}\label{thm_OOG}
When \textit{Assumption \ref{assume_equil_begin}} holds, the optimal value of \eqref{oog_toep} can be obtained by solving the convex SDP
\begin{equation}\label{opti_oog_dual}
\begin{aligned}
\gamma^* = \min \; & \gamma \\
\textrm{s.t.} \quad & \mathcal{T}_{p}^T\mathcal{T}_{p} -  \gamma\mathcal{T}_{r}^T\mathcal{T}_{r} \preceq 0.
\end{aligned}
\end{equation}
\end{theorem}
\begin{proof}
See Appendix
\end{proof}

Thus, the impact assessment problem \eqref{opti_oog_primal} was solved by a convex dual SDP \eqref{opti_oog_dual}. Using S-lemma, we also proved that the duality gap is indeed zero. As stated before, this is the advantage of our work over \cite{umsonst2017security} where it was necessary to solve $N_hn$ optimization problems. Next, we state the necessary and sufficient condition for $\gamma^*$ to be bounded in \textit{Lemma \ref{lemma_bound_T}}.
\begin{lemma}\label{lemma_bound_T}
The value of the optimization problem \eqref{opti_oog_dual} is bounded if and only if $\text{ker}(\mathcal{T}_r) \subseteq \text{ker}(\mathcal{T}_p)$.
\end{lemma}
\begin{proof}
Refer to the proof of \cite[\textit{Lemma 1}]{teixeira2013quantifying}.
\end{proof}

When the value of the optimization problem \eqref{opti_oog_dual} is unbounded, it implies that the adversary can cause huge system disruptions without being detected. The defender would be interested in knowing the answer to: How to alter the system matrices of \eqref{system:CL:uncertain} to result in a bounded impact. The answer is not trivial since the result of \textit{Lemma \ref{lemma_bound_T}} is based on block TMs, which involves higher powers of the matrix $A$. We answer this question without using TMs in the next section through a DP framework.

\section{Dynamic programming framework for attack policy characterization}\label{sec_policy}

In this section, we first adopt the DP framework to determine the attack policy which causes the impact $\gamma^*$. Later, using the results from the DP framework, we discuss the conditions for boundedness of the attack impact. 
\subsection{Optimal attack policy}\label{sub_sec_attack}
Classically, the DP problem determines the attack (input) policy which minimizes a cost function. To this end, using the Lagrange dual function, we rewrite \eqref{opti_oog_primal} as a soft-constrained optimization problem
\begin{equation}\label{lagrange}
\begin{aligned}
\sup_{a\in \ell_{2e,[0,N_h]}} & -\left\{\sum_{k=0}^{N_h}\begin{bmatrix}
x[k]^T\\a[k]^T
\end{bmatrix}^T
\Pi
\begin{bmatrix}
x[k]\\a[k]
\end{bmatrix}\right\} + \gamma^*,
\end{aligned}
\end{equation}
where $\Pi \triangleq \begin{bmatrix}
\gamma^* C_r^TC_r -  C_p^TC_p & \gamma^* C_r^TD_r -  C_p^TD_p\\
\gamma^* D_r^TC_r -  D_p^TC_p & \gamma^* D_r^TD_r -  D_p^TD_p
\end{bmatrix}$. Since strong duality holds, the optimal value of \eqref{opti_oog_primal} and \eqref{lagrange} will be the same. We next proceed to find the optimal attack vector. In \eqref{lagrange}, the optimization variable $a$ does not affect the constant term $\gamma^*$ and thus is dropped. By altering the cost function in \eqref{lagrange}, the $\sup$ is changed to the $\inf$ operator
\begin{equation}\label{indefinite_formulation}
\begin{aligned}
- \inf_{a\in \ell_{2e,[0,N_h]}} & \left\{\sum_{k=0}^{N_h} \begin{bmatrix}
x[k]^T\\a[k]^T
\end{bmatrix}^T
\Pi
\begin{bmatrix}
x[k]\\a[k]
\end{bmatrix}\right\}.
\end{aligned}
\end{equation}

In general, the DP framework assumes that the weighting matrix $\Pi$ of the cost function is positive definite. In our case, $\Pi$ cannot be guaranteed to be positive definite as it depends on $\gamma^*$. To this end, we adopt the principle of DP for indefinite cost function \cite{ferrante2015note}. Before we present the main result, we establish the following assumption:
\begin{assumption}\label{assume_attain}
The value of \eqref{opti_oog_dual} is bounded. $\hfill \triangleleft$
\end{assumption}

Now, we are ready to present the theorem which determines the optimal attack vector based on DP.
\begin{theorem}\label{thm_DP}
The optimal attack vector $a^*[k]$ which minimizes the cost function of \eqref{indefinite_formulation}, is parameterized as a function of an arbitrary attack vector $v[k]$ as
\begin{equation}\label{attack_DP}
    {a}^*[k]=-K_kx[k] + G_kv[k], \;\forall \;0 \leq k \leq N_h
\end{equation}
where $K_k=(R+B_{cl}^TX_{k+1}B_{cl})^{+}(S^T+B_{cl}^TX_{k+1}A_{cl})$, $G_k = I_m - (R+B_{cl}^TX_{k+1}B_{cl})^{+}(R+B_{cl}^TX_{k+1}B)$ and the matrices $X_k,\; \forall\; k \in \{N,\dots,N_h\}$ are obtained by solving the generalized Riccati equation (GRE)
\begin{multline}\label{ricatti}
X_k=Q+A_{cl}^TX_{k+1}A_{cl}-(S+A_{cl}^TX_{k+1}B_{cl})\\\times (R+B_{cl}^TX_{k+1}B_{cl})^{+}(S^T+B_{cl}^TX_{k+1}A_{cl}),
\end{multline}
where $X_{N_h+1}=0$, $Q \triangleq \gamma^* C_r^TC_r -  C_p^TC_p$, $S \triangleq \gamma^* C_r^TD_r -  C_p^TD_p$, $R \triangleq \gamma^* D_r^TD_r -  D_p^TD_p$. Moreover, the  optimal value of \eqref{indefinite_formulation} is given by $-x[0]^TX_0x[0]$.
\end{theorem}
\begin{proof}
Directly follows from \cite[\textit{Theorem 2.1}]{ferrante2015note}.
\end{proof}

\textit{Theorem \ref{thm_DP}} describes a recursive method to calculate the optimal attack vector which minimizes the cost function of \eqref{indefinite_formulation}. Next, we discuss the conditions under which the obtained attack vector is the optimal attack vector to \eqref{opti_oog_primal}. 

Let us characterize the attack vector of \eqref{indefinite_formulation} by \eqref{attack_DP}. Let $x[0]=0$ and $\textbf{v} \triangleq \begin{bmatrix} v[0]^T,\dots,v[N_h]^T \end{bmatrix}^T$. Then \eqref{attack_DP} becomes a function of only the vector $\textbf{v}$ and the system matrices. Let us define $\mathcal{T}_{pv}$ and $\mathcal{T}_{rv}$ such that $\textbf{y}_\textbf{r} = \mathcal{T}_{rv}\textbf{v}$ and $\textbf{y}_\textbf{p} = \mathcal{T}_{pv}\textbf{v}$. Let $\mathcal{A}_k \triangleq A_{cl}-B_{cl}K_k$, then $\mathcal{T}_{\alpha v},\alpha \in \{p,r\}$ is represented in \eqref{toeplitz_rpv}. Then, \textit{Lemma \ref{lemma_primal_attack}} states the necessary and sufficient conditions under which the attack vector obtained from \eqref{attack_DP} is optimal to \eqref{opti_oog_primal}.

\begin{lemma}\label{lemma_primal_attack}
Let $\gamma^*$ be the optimal value of \eqref{opti_oog_dual}. Then, any attack vector of the form $\eqref{attack_DP}$ is optimal to \eqref{opti_oog_primal} if and only if $\textbf{v}^T\mathcal{T}_{rv}^T\mathcal{T}_{rv}\textbf{v} = 1$.
\end{lemma}
\begin{proof}
See Appendix.
\end{proof}

\textit{Lemma \ref{lemma_primal_attack}} states that any attack vector of the form \eqref{attack_DP}, which yields the detection output energy as $1$, is an optimal attack vector to \eqref{opti_oog_primal}.\\[0.5cm]
\begin{strip}
\begin{align}
\mathcal{T}_{\alpha v} = \begin{bmatrix}
D_{\alpha}G_0 & 0  & \dots & 0\\
(C_{\alpha}-D_{\alpha }K_1)B_{cl}G_0 & D_{\alpha }G_1  & \dots & 0\\
\vdots & \vdots & \ddots & \vdots \\
(C_{\alpha}-D_{\alpha}K_{N_h})\prod_{k=N_h-1}^1\mathcal{A}_kB_{cl}G_0 & (C_{\alpha}-D_{\alpha}K_{N_h})\prod_{k=N_h-1}^2\mathcal{A}_kB_{cl}G_1 & \dots & D_{\alpha}G_{N_h}
\end{bmatrix}\label{toeplitz_rpv}
\end{align}
\end{strip}
\subsection{Conditions for bounded attack impact}
In \textit{Section \ref{sub_sec_attack}}, we used the DP framework to determine the optimal attack vector. The state feedback matrices for the attack vector ($K_k$) were obtained from the solution to the GRE \eqref{ricatti}. The necessary and sufficient conditions for a solution to exist for the GRE is stated \textit{Lemma \ref{DP_sol_exist}}.
\begin{lemma}\label{DP_sol_exist}
There exists a solution to the GRE \eqref{ricatti} iff \eqref{con_1} and \eqref{con_2} holds $\forall\; k \in \{1,\dots,N_h\}$. 
\begin{align}
    R+B_{cl}^TX_kB_{cl} &\geq 0, \label{con_1}\\
     \text{ker}(R+B_{cl}^TX_kB_{cl}) &\subseteq \text{ker}(S+A_{cl}^TX_kB_{cl}). \label{con_2}
\end{align}
\end{lemma}
\begin{proof}
Refer to the proof of \cite[\textit{Theorem 2.1}]{ferrante2015note}
\end{proof}

A consequence of \textit{Lemma \ref{DP_sol_exist}} is that, when \eqref{con_1} and \eqref{con_2} does not hold, the value of \eqref{lagrange} (and consequently \eqref{opti_oog_primal} due to strong duality) is unbounded. Thus, the first advantage of the DP framework is that, using the results of \textit{Lemma \ref{DP_sol_exist}}, we can characterize scenarios under which the impact \eqref{opti_oog_primal} becomes unbounded. To this end, we use \textit{Lemma \ref{DP_sol_exist}} to characterize a scenario where the impact is unbounded in \textit{Lemma \ref{limitation}}.
\begin{lemma}\label{limitation}
If $\exists s \neq 0$ such that $D_rs=0$ and $D_ps \neq 0$, then the solution to the GRE \eqref{ricatti} does not exist thus making the impact unbounded.
\end{lemma}
\begin{proof}
See Appendix
\end{proof}

Thus, if the system operator could alter the direct feed-through matrices such that the lemma conditions do not hold, the risk of having an unbounded impact can be, although not eliminated, reduced. Another advantage of using the DP framework is detailed as follows. Let us consider $\gamma^*$ as a variable. If the defender could find a bounded $\gamma^*$ such that the conditions of \textit{Lemma \ref{DP_sol_exist}} hold, the attack impact will be bounded. If the attack impact is not bounded, the defender could alter the system matrices such that the conditions hold. This would result in a bounded/lowered impact. This relationship on altering the system matrices to lower the worst-case impact was not evident from \textit{Lemma \ref{lemma_bound_T}} but now is clearer from the DP framework. A numerical illustration of the proposed approach for stealthy attack design is provided in the next section.


\section{Numerical Example}\label{sec_example}

Consider a power generating system \cite[Section 4]{park2019stealthy} as represented by \eqref{power_AB} and \eqref{power_C}. 
\begin{align}
\label{power_AB} \begin{bmatrix}
\dot{\eta}_1\\ \dot{\eta}_2 \\ \dot{\eta}_3
\end{bmatrix} &= 
\begin{bmatrix}
\frac{-1}{T_{lm}} & \frac{K_{lm}}{T_{lm}} & \frac{-2K_{lm}}{T_{lm}}\\
0 & \frac{-2}{T_h} & \frac{6}{T_h}\\
\frac{-1}{T_g R} & 0 & \frac{-1}{T_g}
\end{bmatrix}
\underbrace{\begin{bmatrix}
{\eta}_1\\ {\eta}_2 \\ {\eta}_3
\end{bmatrix}}_{\eta}
+ \begin{bmatrix}
0\\ 0 \\ \frac{1}{T_g}
\end{bmatrix}
\Tilde{u}
\end{align}
\begin{align}
\label{power_C} y &= \underbrace{ \begin{bmatrix}
1 & 0 & 0 
\end{bmatrix}}_{C}\begin{bmatrix}
\eta_1\\ \eta_2 \\ \eta_3
\end{bmatrix},\;\;
y_p = \underbrace{
\begin{bmatrix}
1 & 0 & 0\\ 0 & 1 & 0
\end{bmatrix}}_{C_p}\begin{bmatrix}
\eta_1\\ \eta_2 \\ \eta_3
\end{bmatrix}
\end{align}

Here $\eta \triangleq [df; dp + 2 dx; dx]$, $df$ is the frequency deviation in \mbox{Hz}, $dp$ is the change in the generator output per unit (\mbox{p.u.}), and $dx$ is the change in the valve position \mbox{p.u.}. The constants $T_{lm}=6, T_h=4$, and $T_g=0.2$ represent the time constants of load and machine, hydro turbine, and governor, respectively, and $R = 0.05 (Hz/p.u.)$ is the speed regulation due to the governor action. The constant $K_{lm}=1$ represents the steady state gain of the load and machine. The DT system matrices are obtained by discretizing the process using zero-order hold with a sampling time $T_s=0.1 \;\mbox{seconds}$. The DT process is stabilized with an output feedback controller of the form \eqref{C} with $D_c=19$. The detector is of the form \eqref{D} where $A_e = (A_d-K_eC_d), B_e = B_d, C_e=C_d$ and $K_e=\begin{bmatrix} 0.17 & -2.83 & -7.43 \end{bmatrix}^T$. The adversary attacks only the actuator, i.e.: $B_a = 1$ and $D_a =0$. The other unspecified matrices are zero. The system is assumed to satisfy $\eta[0]=0$. We consider a horizon length of $N_h=50$.

By solving the optimization problem \eqref{opti_oog_dual}, we obtain $\gamma^* = 4733.3$. We formulate the Lagrange dual function similar to \eqref{indefinite_formulation}. The matrices $K_k, G_k$ and $X_k$ were obtained by solving the GRE described in \textit{Theorem \ref{thm_DP}}. Using these matrices, we obtained the matrices $\mathcal{T}_{pv}$ and $\mathcal{T}_{rv}$.  According to \textit{Lemma \ref{lemma_primal_attack}}, we found the eigenvector $\textbf{v}$, corresponding to the eigenvalue $\gamma^*$ of the matrix pencil $(\mathcal{T}_{pv}^T\mathcal{T}_{pv},\mathcal{T}_{rv}^T\mathcal{T}_{rv})$. This vector is scaled such that $\textbf{v}^{T}\mathcal{T}_{rv}^T\mathcal{T}_{rv}\textbf{v} = 1$. The obtained optimal attack vector is
\begin{equation}
    v^*[k] = \begin{cases} 
      33.9024, & k = 0 \\
      0, & \text{otherwise}.
      \end{cases}
\end{equation}

The resulting attack vector obtained from \eqref{attack_DP} is shown in Fig. \ref{attack_fig}. Applying this attack signal to \eqref{system:CL:uncertain}, the performance of the system is shown in Fig \ref{energy_fig}. It can be seen that the detection output energy reaches the value $1$ which represents that the constraint of the primal problem \eqref{opti_oog_primal} is satisfied. Similarly, the performance output energy reaches the value of $4733.3$ which shows that the duality gap is zero. Finally, in this example, the matrices $D_p$ and $D_r$ are zero and the limitation described in \textit{Lemma \ref{limitation}} does not occur.

\begin{figure}
    \centering
    \includegraphics[width=8.4cm]{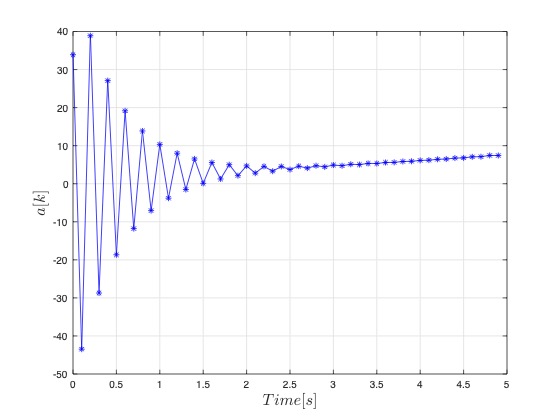}
    \caption{Optimal attack vector}
    \label{attack_fig}
\end{figure}
\begin{figure}
    \centering
    \includegraphics[width=8.4cm]{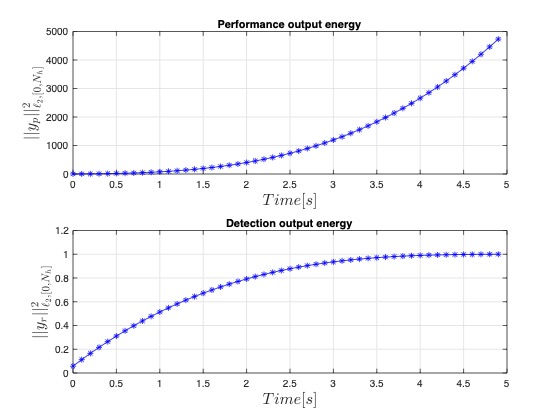}
    \caption{Performance and detection output energies}
    \label{energy_fig}
\end{figure}

\section{Conclusion}\label{sec_conclusion}

In this paper, we considered FDI attacks which aim at maximizing impact while staying undetected in the finite horizon. We formulated the impact assessment problem and it corresponded to a non-convex optimization problem. This problem was shown to be equivalent, using S-lemma, to a convex dual SDP. Secondly, we formulated a soft-constrained optimization problem using the Lagrangian dual function, to determine the optimal attack vector. The framework of DP was used to determine the optimal attack vector. We also provided the necessary and sufficient conditions under which the obtained attack vector is optimal to the primal problem. Finally, the results were illustrated through numerical examples. Future works include extending the framework to adaptive dynamic programming.

\appendix
\section*{Proof of \textit{Theorem \ref{thm_OOG}}}
Before presenting the proof, an introduction to S-lemma is provided 
\begin{lemma}\cite[\textit{Theorem 2}]{yakubovich1997s}\label{S-lemma}
Let $q_a(x)$ and $q_b(x)$ be quadratic functions. Suppose $\exists \; \bar{x}$ such that $q_a(\bar{x}) \geq 0$. Then the quadratic inequality $q_b(x) \geq 0$ is a consequence of $q_a(x) \geq 0$ if and only if $q_b(x) \geq \gamma q_a(x)$. $\hfill \square$
\end{lemma}
\begin{proof}[Proof of \textit{Theorem \ref{thm_OOG}}]
When \textit{Assumption \ref{assume_equil_begin}} hold, we can rewrite \eqref{oog_toep} as
\begin{equation}\label{inter_2}
\begin{aligned}
-\inf_{t} \; & -\textbf{a}^T\mathcal{T}_p^T\mathcal{T}_p\textbf{a} \\
\textrm{s.t.} \quad &  \textbf{a}^T\mathcal{T}_r^T\mathcal{T}_r\textbf{a} \leq 1
\end{aligned}
\end{equation}

Using the hypo-graph formulation, \eqref{inter_2} can be recast as
\begin{equation}\label{s5}
\begin{aligned}
-\max_{t} \; & t \\
\textrm{s.t.} \quad & -\textbf{a}^T\mathcal{T}_p^T\mathcal{T}_p\textbf{a} \geq t\; \text{whenever}\; \textbf{a}^T\mathcal{T}_r^T\mathcal{T}_r\textbf{a} \leq 1,
\end{aligned}
\end{equation}
which can again be rewritten as
\begin{equation}\label{duality_proof_ref_1}
\begin{aligned}
-\max_{t} & \;t \\
\textrm{s.t.}& 1 - \textbf{a}^T\mathcal{T}_r^T\mathcal{T}_r\textbf{a} \geq 0 \implies -\textbf{a}^T\mathcal{T}_p^T\mathcal{T}_p\textbf{a}-t \geq 0.
\end{aligned}
\end{equation}

Let us define the quadratic functions $q_a(\textbf{a}) \triangleq 1- \textbf{a}^T\mathcal{T}_r^T\mathcal{T}_r\textbf{a}$ and $q_b(\textbf{a}) \triangleq -\textbf{a}^T\mathcal{T}_p^T\mathcal{T}_p\textbf{a} - t$. Firstly, if the optimization problem \eqref{oog_toep} is feasible, then $\exists\; \bar{\textbf{a}}$ such that $q_a(\bar{\textbf{a}}) \geq 0$ is feasible. Then from \textit{Lemma \ref{S-lemma}}, it holds that $ q_a(\textbf{a}) \geq 0 \implies q_b(\textbf{a}) \geq 0$ if and only if $q_b(\textbf{a}) \geq \gamma q_a(\textbf{a})$. Using this iff relation, \eqref{duality_proof_ref_1} can be reformulated as
\begin{equation}
\begin{aligned}
-\max_{t,\gamma} \; & t \\
\textrm{s.t.} \quad & q_b(\textbf{a}) \geq \gamma q_a(\textbf{a}).
\end{aligned}
\end{equation}

Substituting the definition of the quadratic function yields
\begin{equation}\label{f1}
\begin{aligned}
-\max_{t\gamma} \; & t \\
\textrm{s.t.} \quad &  -\textbf{a}^T\mathcal{T}_p^T\mathcal{T}_p\textbf{a} - \gamma + \gamma \textbf{a}^T\mathcal{T}_r^T\mathcal{T}_r\textbf{a} \geq t.
\end{aligned}
\end{equation}

Since \eqref{f1} resembles an epigraph formulation, it can be rewritten as 
\begin{equation}\label{duality_proof_ref_2}
\begin{aligned}
-\max_{\gamma} \{ \underbrace{\min_{\textbf{a}}  \{-\textbf{a}^T\mathcal{T}_p^T\mathcal{T}_p\textbf{a} - \gamma + \gamma \textbf{a}^T\mathcal{T}_r^T\mathcal{T}_r\textbf{a}  \}}_{\kappa}\}.
\end{aligned}
\end{equation}

Since $\kappa$ is a minimization problem, it holds that
\[
\kappa = 
\begin{cases}
 -\gamma, & \text{iff}\;\; -\mathcal{T}_p^T\mathcal{T}_p + \gamma \mathcal{T}_r^T\mathcal{T}_r \geq 0\\
    -\infty,  & \text{otherwise}
\end{cases}.
\]

Therefore \eqref{duality_proof_ref_2} can be rewritten as
\begin{equation}
\begin{aligned}
\min \; & \gamma\\
\textrm{s.t.} \quad & -\mathcal{T}_p^T\mathcal{T}_p + \gamma \mathcal{T}_r^T\mathcal{T}_r  \geq 0,
\end{aligned}
\end{equation}
which concludes the proof.
\end{proof}

\section*{Proof of \textit{Lemma \ref{lemma_primal_attack}}}
Before presenting the proof, an introduction to optimality conditions is presented for general quadratically constrained quadratic problem (QCQP)
\begin{lemma}\label{KKT}\cite[\textit{ Proposition 3.3}]{jeyakumar2007non}
Consider the inequality constrained QCQP
\begin{equation}\label{opti_KKT}
\begin{aligned}
\min \; &  x^TA_0x\\
\textrm{s.t.} \quad & x^TA_1x + c_1 \leq 0,
\end{aligned}
\end{equation}
where $A_0, A_1$ are symmetric matrices and $c_1 \in \mathbb{R}$. Suppose $\exists \; x_0$ such that $x_0^TA_1x_0 + c_1 < 0$. Then $x_*$ is a global minimizer of \eqref{opti_KKT} if and only if $\exists \;x_{*}$ and $\lambda_*$ such that the following Karush–Kuhn–Tucker (KKT) conditions hold
\begin{enumerate}
    \item Primal feasibility: $x_*^TA_1x_* + c_1 \leq 0$.
    \item Dual feasibility: $\lambda_* \geq 0$.
    \item Complementary slackness: $\lambda_*(x^TA_1x + c_1) =0$.
    \item Stationarity: $ (\lambda_*A_1 + A_0)x_* =0$.
    \item $(\lambda_*A_1 + A_0) \succeq 0$ $\hfill \square$
\end{enumerate}
\end{lemma}
\begin{proof}[Proof of \textit{Lemma \ref{lemma_primal_attack}}]
Consider the optimization problem \eqref{opti_oog_primal} which can be reformulated as \eqref{inter_2}. To use the result of \textit{Lemma \ref{KKT}}, let $A_0 = - \mathcal{T}_p^T\mathcal{T}_p, A_1 = \mathcal{T}_r^T\mathcal{T}_r$ and $c_1 = -1$. We know that $\exists\; \textbf{a}_0 \triangleq 0$, such that $\textbf{a}_0^T\mathcal{T}_r^T\mathcal{T}_r\textbf{a} - 1 <0$. It then follows that any primal argument (the attack vector $\textbf{a}$) satisfying all the KKT conditions, along with a dual argument (the Lagrange multiplier $\gamma^*$), is a globally optimal primal attack vector. 

We show in the proof below that, when $\gamma^*$ is the optimal value of \eqref{opti_oog_dual}, any attack vector of the form \eqref{attack_DP}, which fulfils the conditions of \textit{Lemma \ref{lemma_primal_attack}}, satisfies all the KKT conditions and thus proving the statement of \textit{Lemma \ref{lemma_primal_attack}}. To begin with, with an abuse of notation, let the stacked attack vector resulting from \eqref{attack_DP} be represented by \textbf{a}.

\textit{Primal feasibility} requires that 
\begin{equation}\label{proof_KKT_1}
\textbf{a}^T\mathcal{T}_r^T\mathcal{T}_r\textbf{a} \leq 1.   
\end{equation}
Given that the attack vector is of the form \eqref{attack_DP}, \eqref{proof_KKT_1} can be rewritten as $\textbf{v}^T\mathcal{T}_{rv}^T\mathcal{T}_{rv}\textbf{v} \leq 1$. From the lemma statement, we know that $\textbf{v}^T\mathcal{T}_{rv}^T\mathcal{T}_{rv}\textbf{v} = 1$.

\textit{Dual feasibility} requires that $\gamma^* \geq 0$. This is satisfied since it is a constraint to the optimization problem \eqref{opti_oog_dual}.

\textit{Complementary slackness} holds if $\gamma^*(\textbf{a}^T\mathcal{T}_r^T\mathcal{T}_r\textbf{a} - 1)= 0$. We showed in KKT condition 1 that $\textbf{a}^T\mathcal{T}_r^T\mathcal{T}_r\textbf{a}-1 = 0$. Thus complementary slackness holds.

\textit{Stationarity} requires that $ (\gamma^*\mathcal{T}_r^T\mathcal{T}_r-\mathcal{T}_p^T\mathcal{T}_p)\textbf{a}=0$. Simplifying it further, we obtain
\begin{equation}\label{KKT_4_s1}
(\gamma^*\mathcal{T}_{rv}^T\mathcal{T}_{rv}-\mathcal{T}_{pv}^T\mathcal{T}_{pv})\textbf{v}=0.
\end{equation}
To this end, consider the term $\textbf{v}^T(\gamma^*\mathcal{T}_{rv}^T\mathcal{T}_{rv}-\mathcal{T}_{pv}^T\mathcal{T}_{pv})\textbf{v}$. This term can be rewritten as the cost function of \eqref{indefinite_formulation}. From \textit{Theorem \ref{thm_DP}}, we know that the optimal value of the cost function can be characterized as $-x[0]^TX_0x[0]$. Since $x[0] = 0$, \eqref{KKT_4_s2} holds from which \eqref{KKT_4_s1} follows. 
\begin{equation}\label{KKT_4_s2}
\textbf{v}^T(\gamma^*\mathcal{T}_{rv}^T\mathcal{T}_{rv}-\mathcal{T}_{pv}^T\mathcal{T}_{pv})\textbf{v}=\textbf{a}^T(\gamma^*\mathcal{T}_r^T\mathcal{T}_r-\mathcal{T}_p^T\mathcal{T}_p)\textbf{a}=0.
\end{equation}

\textit{KKT condition 5} requires that  requires that
\begin{equation}\label{KKT_p1}
\gamma^*\mathcal{T}_r^T\mathcal{T}_r - \mathcal{T}_p^T\mathcal{T}_p \succeq 0.
\end{equation}
Since \eqref{KKT_p1} is a constraint of \eqref{opti_oog_dual}, KKT condition 5 holds. We have thus proven that that the attack vector of the form \eqref{attack_DP}, when satisfying the condition of \textit{Lemma \ref{lemma_primal_attack}}, is a global maximizer to the optimization problem \eqref{oog_toep} (or minimizer to the optimization problem \eqref{inter_2}). This concludes the proof.
\end{proof}

\section*{Proof of \textit{Lemma \ref{limitation}}}
\begin{proof}
Let us assume that $\exists s \neq 0$ such that $D_rs=0$ and $D_ps \neq 0$. From \textit{Theorem \ref{thm_DP}}, we know that $R = \gamma^*D_r^TD_r-D_p^TD_p$ and $X_{N_h} = 0$. Then, from \textit{Lemma \ref{DP_sol_exist}}, for a solution to exist for the GRE, \eqref{con_1} should hold. Let $k=N_h$, then using the definition of $R$ and $X_{N_h}$ in \eqref{con_1} yields
\begin{equation}\label{s3}
    R+B_{cl}^TX_kB_{cl} = \gamma^*D_r^TD_r-D_p^TD_p \succeq 0.
\end{equation}
By the definition of positive definiteness, \eqref{s3} can be equivalently written as
\begin{equation}\label{s4}
    x^T(\gamma^*D_r^TD_r-D_p^TD_p)x \geq 0 \;\; \forall x \in \mathbb{R}^{n_a}.
\end{equation}
Let $x=s$ and from the lemma statement we know that $s^TD_r^TD_rs =0$. Therefore, for \eqref{s4} to hold, it should hold that $-sD_p^TD_ps \geq 0$. This inequality cannot hold since $sD_p^TD_ps > 0$. This concludes the proof.
\end{proof}

\bibliographystyle{ieeetr}
\bibliography{ref}
\clearpage
\end{document}